\documentclass[graybox]{svmult}

\usepackage{mathptmx}       
\usepackage{helvet}         
\usepackage{courier}        
\usepackage{type1cm}        
\usepackage{makeidx}         
\usepackage{graphicx}        
\usepackage{multicol}        
\usepackage[bottom]{footmisc}

\usepackage{amssymb}
\usepackage{amsmath}

\renewcommand\E{\mathbb{E}}
\newcommand\Z{\mathbb{Z}}
\newcommand\R{\mathbb{R}}
\newcommand\eps{\varepsilon}
\newcommand\Exp{\operatorname{Exp}}

\makeindex             

\begin{document}

\title*{Narrow progressions in the primes}
\author{Terence Tao and Tamar Ziegler}
\institute{Terence Tao \at UCLA Department of Mathematics, 405 Hilgard Ave, Los Angeles CA 90095, USA, \email{tao@math.ucla.edu}
\and Tamar Ziegler \at Einstein Institute of Mathematics, Edmond J. Safra Campus, Givat Ram, The Hebrew University of Jerusalem, Jerusalem, 91904, Israel \email{tamarz@math.huji.ac.il}}
%
%
\maketitle

\abstract{In a previous paper of the authors, we showed that for any polynomials $P_1,\dots,P_k \in \Z[\mathbf{m}]$ with $P_1(0)=\dots=P_k(0)$ and any subset $A$ of the primes in $[N] = \{1,\dots,N\}$ of relative density at least $\delta>0$, one can find a ``polynomial progression'' $a+P_1(r),\dots,a+P_k(r)$ in $A$ with $0 < |r| \leq N^{o(1)}$, if $N$ is sufficiently large depending on $k,P_1,\dots,P_k$ and $\delta$.  In this paper we shorten the size of this progression to $0 < |r| \leq \log^L N$, where $L$ depends on $k,P_1,\dots,P_k$ and $\delta$.  In the linear case $P_i = (i-1)\mathbf{m}$, we can take $L$ independent of $\delta$.
The main new ingredient is the use of the densification method of Conlon, Fox, and Zhao to avoid having to directly correlate the enveloping sieve with dual functions of unbounded functions.}

\section{Introduction}

\subsection{Previous results}

We begin by recalling the well-known theorem of Szemer\'edi \cite{szemeredi} on arithmetic progressions, which we phrase as follows:

\begin{theorem}[Szemer\'edi's theorem]\label{szthm}  Let $k \geq 1$ and $\delta > 0$, and suppose that $N$ is sufficiently large depending on $k,\delta$.  Then any subset $A$ of $[N] := \{n \in \Z: 1 \leq n \leq N \}$ with cardinality $|A| \geq \delta N$ will contain at least one arithmetic progression $a, a+r, \ldots, a+(k-1)r$ of length $k$, with $r > 0$.
\end{theorem}

In fact, by partitioning $[N]$ into intervals of some sufficiently large but constant size $L = L(k,\delta)$ and using the pigeonhole principle, one can ensure that the progression described above is ``narrow'' in the sense that $r \leq L(k,\delta)$.

In \cite{gt-primes}, Szemer\'edi's theorem was relativized to the primes $\mathbb{P} = \{2,3,5,7,\dots\}$:

\begin{theorem}[Szemer\'edi's theorem in the primes]\label{szp}  Let $k \geq 2$ and $\delta > 0$, and suppose that $N$ is sufficiently large depending on $k,\delta$.  Then any subset $A$ of $[N] \cap \mathbb{P}$ with $|A| \geq \delta |[N] \cap \mathbb{P}|$ will contain at least one arithmetic progression $a, a+r, \ldots, a+(k-1)r$ of length $k$, with $r > 0$.
\end{theorem}

In particular, the primes contain arbitrarily long arithmetic progressions.   

The proof of Theorem \ref{szp} does not place a bound on $r$ beyond the trivial bound $r \leq N$.  In contrast with Theorem \ref{szthm}, one cannot hope here to make the step size $r$ of the progression as short as $L(k,\delta)$.  Indeed, as observed in \cite{benatar}, if $N$ is large enough then from \cite[Theorem 3]{gpy} one can find a subset $A$ of $[N] \cap \mathbb{P}$ with $|A| \geq \frac{1}{2} |[N] \cap \mathbb{P}|$ (say) such that the gap between any two consecutive elements of $A$ is $\geq c\log N$ for some absolute constant $c>0$, which of course implies in this case that $r$ must be at least $c \log N$ as well.  Indeed, one can improve this lower bound to $\log^{k-1} N$ by a small modification of the argument:

\begin{proposition}\label{rrn}  Let $k \geq 2$, let $\eps >0$ be sufficiently small depending on $k$, and suppose that $N$ is sufficiently large depending on $k$.  Then there exists a subset $A$ of $[N] \cap \mathbb{P}$ with $|A| \gg |[N] \cap \mathbb{P}|$ such that $A$ does not contain any arithmetic progression $a, a+r, \ldots, a+(k-1)r$ of length $k$ with $0 < r < \eps \log^{k-1} N$.
\end{proposition}

\begin{proof}  For any $0 < r < \eps \log^{k-1} N$, let $N_r$ denote the number of $a \in [N]$ such that $a,a+r,\dots,a+(k-1)r$ are all prime.  By the union bound and the prime number theorem, it will suffice to show that
$$ \sum_{0 < r < \eps \log^{k-1} N} N_r \leq \frac{1}{2} \frac{N}{\log N}$$
(say), since one can then form $A$ by removing all the elements $a$ in $[N] \cap \mathbb{P}$ that are associated to one of the $N_r$.  But from standard sieve theoretic bounds (see e.g. \cite[Theorem 5.7]{hr}) we have
$$ N_r \leq C_k \mathfrak{G}(k,r) \frac{N}{\log^k N}$$
where $C_k$ depends only on $k$ and $\mathfrak{G}(k,r)$ is the singular series\footnote{All sums and products over $p$ in this paper are understood to be ranging over primes.}
\begin{equation}\label{grdef}
 \mathfrak{G}(k,r) = \prod_p \left(1-\frac{1}{p}\right)^k \left(1 - \frac{\nu_p(r)}{p}\right)
\end{equation}
and $\nu_p(r)$ is the number of residue classes $a \in \Z/p\Z$ such that at least one of $a,a+r,\dots,a+(k-1)r$ is equal to $0$ mod $p$, so it suffices to show that
$$ \sum_{0 < r < M} \mathfrak{G}(k,r) \leq C_k M$$
for any $M \geq 1$ (we allow $C_k$ to represent a different $k$-dependent constant from line to line, or even within the same line).  One could obtain precise asymptotics on the left-hand side using the calculations of Gallagher \cite{gallagher}, but we can obtain a crude upper bound that suffices as follows.  From \eqref{grdef} we see that
$$ \mathfrak{G}(k,r) \leq C_k \exp\left( C_k \sum_{p|r} \frac{1}{p} \right)$$
and hence by \cite[Lemma E.1]{tz}
$$ \mathfrak{G}(k,r) \leq C_k \sum_{p|r} \frac{\log^{C_k} p}{p}$$
and thus
$$ \sum_{0 < r < M} \mathfrak{G}(k,r) \leq C_k \sum_p  \frac{\log^{C_k} p}{p} \frac{M}{p} \leq C_k M$$
as required.
\end{proof}

In the converse direction, if we use the ``Cram\'er random model'' of approximating the primes $\mathbb{P} \cap [N]$ by a random subset of $[N]$ of density $1/\log N$, we can asymptotically almost surely match this lower bound, thanks to the work of Conlon-Gowers \cite{cg} and Schacht \cite{schacht}:

\begin{proposition}\label{nrr}  Let $k \geq 2$ and $\delta, \eps > 0$, let $C>0$ be sufficiently large depending on $\delta,k$, and suppose that $N$ is sufficiently large depending on $k,\delta,\eps$.  Let $P \subset [N]$ be chosen randomly, such that each $n \in [N]$ lies in $P$ with an independent probability of $1/\log N$.  Then with probability at least $1-\eps$, every subset $A$ of $P$ with $|A| \geq \delta |P|$ will contain an arithmetic progression $a, a+r, \ldots, a+(k-1)r$ of length $k$ with $0 < r \leq C \log^{k-1} N$.
\end{proposition}

We remark that a modification of the argument in Proposition \ref{rrn} shows that we cannot replace the large constant $C$ here by an arbitrarily small constant $c>0$.

\begin{proof}  We partition $[N]$ into intervals $I_1,\dots,I_m$ of length between $\frac{C}{2} \log^{k-1} N$ and $C\log^{k-1} N$, thus $m \leq \frac{2 N}{C \log^{k-1} N}$.  For each interval $I_i$, we see from \cite[Theorem 2.2]{schacht} or \cite[Theorem 1.12]{cg} that with probability at least $1-\frac{\delta \eps}{10}$, every subset $A_i$ of $P \cap I_i$ with $|A_i| \geq \frac{\delta}{2} |P \cap I_i|$ will contain an arithmetic progression of length at least $k$.  Call an interval $I_i$ \emph{bad} if this property does not hold, thus each $I_i$ is bad with probability at most $\delta \eps/10$.  By linearity of expectation followed by Markov's inequality, we conclude that with probability at least $1-\eps$, at most $\frac{\delta N}{5C \log^{k-1} N}$ of the $I_i$ are bad.  Then if $A \subset P$ is such that $|A| \geq \delta |P|$, then at most $\frac{\delta}{2} |P|$ of the elements of $A$ are contained in bad intervals, so from the pigeonhole principle there exists a good interval $I_i$ such that $|A \cap I_i| \geq \frac{\delta}{2} |P \cap I_i|$, and the claim follows.
\end{proof}

It is thus natural to conjecture that in Theorem \ref{szp} one can take $r$ to be as small as $\log^{k-1+o(1)} N$.

\begin{remark} If one seeks progressions inside the full set $\mathbb{P}$ of primes, rather than of dense subsets of the primes, then the Hardy-Littlewood prime tuples conjecture \cite{hardy} predicts that one can take $r$ to be of size $O_k(1)$; indeed, one should be able to take $r$ to be the product of all the primes less than or equal to $k$.  In the case $k=2$, the claim that one can take $r=O(1)$ amounts to finding infinitely many bounded gaps between primes, a claim that was only recently established by Zhang \cite{zhang}.  For higher $k$, the claim $r=O_k(1)$ appears to currently be out of reach of known methods; the best known result in this direction, due to Maynard \cite{maynard} (and also independently in unpublished work of the first author), shows that for any sufficiently large $R>1$, there exist infinitely many intervals of natural numbers of length $R$ that contain $\geq c \log R$ primes for some absolute constant $c>0$, but this is too sparse a set of primes to expect to find length $k$ progressions for any $k \geq 3$.
\end{remark}

We now consider generalizations of the above results, in which arithmetic progressions $a,a+r,\dots,a+(k-1)r$ are replaced by ``polynomial progressions'' $a+P_1(r),\dots,a+P_k(r)$.  More precisely, let $\Z[\mathbf{m}]$ denote the ring of polynomials of one indeterminate variable $\mathbf{m}$ with integer coefficients.  Then Bergelson and Leibman \cite{bl} established the following polynomial version of Theorem \ref{szthm}:

\begin{theorem}[Polynomial Szemer\'edi's theorem]\label{pzt} Let $k \geq 1$, let $P_1,\dots,P_k \in \Z[\mathbf{m}]$ be such that $P_1(0)=\dots=P_k(0)$, let $\delta > 0$, and suppose that $N$ is sufficiently large depending on $k,P_1,\dots,P_k,\delta$.  Then any subset $A$ of $[N]$ with cardinality $|A| \geq \delta N$ will contain at least one polynomial progression $a+P_1(r), a+P_2(r), \ldots, a+P_k(r)$ with $r > 0$.
\end{theorem}

Of course, Theorem \ref{szthm} is the special case of Theorem \ref{pzt} when $P_i = (i-1) \mathbf{m}$. As with Theorem \ref{szthm}, a partitioning argument shows that we may take $r \leq L(k,P_1,\dots,P_k,\delta)$ for some quantity $L$ depending on the indicated parameters.  The polynomial analogue of Theorem \ref{szp} was established by the authors in \cite{tz}:

\begin{theorem}[Polynomial Szemer\'edi's theorem in the primes]\label{szpp}  Let $k \geq 2$, let $P_1,\dots,P_k \in \Z[\mathbf{m}]$ be such that $P_1(0)=\dots=P_k(0)$,  $\eps, \delta > 0$, and suppose that $N$ is sufficiently large depending on $k,P_1,\dots,P_k,\delta,\eps$.  Then any subset $A$ of $[N] \cap \mathbb{P}$ with $|A| \geq \delta |[N] \cap \mathbb{P}|$ will contain at least one polynomial progression $a+P_1(r), a+P_2(r), \ldots, a+P_k(r)$ with $0 < r < N^\eps$.
\end{theorem}

In particular, this implies Theorem \ref{szp} with a bound $0 < r \leq N^{o(1)}$.  

\begin{remark}  The condition $P_1(0)=\dots=P_k(0)$ in Theorem \ref{pzt} can be relaxed to the property of \emph{intersectivity} (that $P_1,\dots,P_k$ have a common root in the profinite integers $\hat \Z = \lim_{\infty \leftarrow m} \Z/m\Z$); see \cite{bll}.  It is not yet known if Theorem \ref{szpp} can similarly be relaxed to intersective polynomials, except in the $k=2$ case which was established in \cite{le}. We will not pursue this matter here.
\end{remark}

\subsection{Main new result}

Our main result is to improve the bound on the $r$ parameter in Theorem \ref{szpp} to be polylogarithmic in size:

\begin{theorem}[Short polynomial progressions in the primes]\label{szpp-short}  Let $k \geq 2$, let $P_1,\dots,P_k \in \Z[\mathbf{m}]$ be such that $P_1(0)=\dots=P_k(0)$, $\eps, \delta > 0$, and suppose that $N$ is sufficiently large depending on $k,P_1,\dots,P_k,\delta,\eps$.  Then any subset $A$ of $[N] \cap \mathbb{P}$ with $|A| \geq \delta |[N] \cap \mathbb{P}|$ will contain at least one polynomial progression $a+P_1(r), a+P_2(r), \ldots, a+P_k(r)$ with $0 < r < \log^L N$, where $L$ depends only on $k,P_1,\dots,P_k,\delta$.
\end{theorem}

In particular, there are infinitely polynomial progressions $a+P_1(r), a+P_2(r), \ldots, a+P_k(r)$ consisting entirely of primes with $0 < r \ll \log^L a$, with $L$ now depending only on $k,P_1,\dots,P_k$.  This is new even in the case of arithmetic progressions $a, a+r, \dots, a+(k-1)r$.

A modification of the proof of Proposition \ref{rrn} shows that some power of $\log N$ is needed in the upper bound on $r$ in Theorem \ref{szpp-short}.  However, we do not know what the optimal value of $L$ is; our argument for general $P_1,\dots,P_k$ uses the PET induction method \cite{pet}, and as such $L$ will grow rapidly with the degrees of the $P_1,\dots,P_k$.  The dependence of $L$ on $\delta$ occurs for technical reasons, and we conjecture that one can in fact select $L$ to be independent of $\delta$; we can verify this conjecture in the arithmetic progression case $P_i = (i-1) \mathbf{m}$, and in fact we can take the explicit value $L := Ck 2^{k}$ in this case for some fixed constant $C$ (actually $C=3$ would already suffice).  We discuss this explicit variant of Theorem \ref{szpp-short} in Section \ref{linear-sec}.  Propositions \ref{rrn}, \ref{nrr} suggest that we should in fact be able to set $L = k-1+\eps$ in these cases, although our methods do not seem strong enough to achieve this, even in the $k=3$ case.  

It is possible that one might be able to directly modify of the arguments in \cite{tz} (which are in turn based on those in \cite{gt-primes})  to establish Theorem \ref{szpp-short}; the main technical difficulties in doing so are the need of having to correlate the enveloping sieve with dual functions of unbounded functions and verifying  the ``correlation condition'' required in that argument, in the setting when $r$ is as small as $\log^L N$ (and the argument appears to have no chance of working when $L$ is independent of $\delta$).  On the other hand, the need for the analogous correlation and boundedness conditions in \cite{gt-primes} to prove Theorem \ref{szp} were recently removed\footnote{The ``dual function condition'' that the dual function of the enveloping sieve is bounded already failed in the arguments in \cite{tz}, which was a significant cause of the complexity of that paper due to the need to find substitutes for this condition (in particular, the correlation condition became significantly more difficult to even state, let alone prove).  But the arguments in \cite{cfz} do not require any version of the dual function condition at all, leading to some simplifications in the current argument over those in \cite{tz}.} 
by Conlon, Fox, and Zhao \cite{cfz}, using a new method which they refer to as ``densification''.  We will be able to combine the densification method with the arguments in \cite{tz} to establish Theorem \ref{szpp-short}.  As a consequence, we also obtain a slightly different proof of Theorem \ref{szpp} than the one in \cite{tz}, in which the (rather complicated) verification of the correlation condition is no longer necessary, but the densification arguments of Conlon, Fox, and Zhao are inserted instead.

\section{Preliminary reductions}\label{notation-sec}

We now begin the proof of Theorem \ref{szpp-short}.  
We use the following asymptotic notation.  We let $N'$ be an asymptotic parameter tending to infinity along some sequence $N' = N'_j$ of natural numbers.  All mathematical objects in this paper are implicitly permitted to depend on $N'$, unless they are explicitly designated to be \emph{fixed}, in which case they are independent of $N'$.  We use $X = O(Y)$, $X \ll Y$, or $Y \gg X$ to denote the estimate $|X| \leq CY$ for some fixed $C$, and $X = o(Y)$ to denote the estimate $|X| \leq c(N') Y$ where $c(N') \to 0$ as $N' \to \infty$.  Our statements will be implicitly restricted to the regime in which $N'$ is sufficiently large depending on all fixed parameters.

Suppose for sake of contradiction that Theorem \ref{szpp-short} failed.  Carefully negating the quantifiers (and relabeling $N$ as $N'$, for reasons that will be clearer later), we conclude that we may find a fixed $k \geq 2$ and fixed polynomials $P_1,\dots,P_k \in \Z[\mathbf{m}]$ with $P_1(0)=\dots=P_k(0)$, fixed $\delta > 0$, a sequence $N' = N'_j$ of natural numbers going to infinity, and a set $A = A_{N'} \subset [N',2N'] \cap \mathbb{P}$ such that
$$ |A| \geq \delta |[N'] \cap \mathbb{P}|$$
and such that for any fixed $L>0$, there are no polynomial progressions $a+P_1(r), a+P_2(r), \ldots, a+P_k(r)$ in $A$ with $0 < r < \log^L N'$ (recall that we assume $N'$ sufficiently large depending on fixed quantities such as $L$).

The first few reductions are essentially the same to those in \cite{tz}.  We begin with the ``$W$-trick'' from \cite{gt-primes} to eliminate local irregularities modulo small primes.  We let $w = w_{N'}$ be a sufficiently slowly growing function of $N'$; for instance, we could take $w := \frac{1}{10} \log\log\log N'$ as in \cite{tz} for sake of concreteness, although the precise value of $w$ is unimportant.   We then define the quantity
$$ W := \prod_{p<w} p$$
and the natural number
$$ N := \left\lfloor \frac{N'}{W} \right\rfloor.$$
From the prime number theorem\footnote{Actually, the weaker lower bound $\pi(x) \gg \frac{x}{\log x}$ of Chebyshev would suffice here.} we have
$$ |A| \gg \frac{N W}{\log N}$$ 
and all elements of $A$ larger than $\sqrt{N}$ (say) are coprime to $W$.  Thus, by the pigeonhole principle, we may find $b \in [W]$ coprime to $W$ (and depending on $N'$ of course) such that
$$ |\{ n \in [N]: nW+b \in A \}| \gg \frac{N W}{\phi(W) \log N},$$
where $\phi$ denotes the Euler totient function.

The domain $[N]$ is not quite translation invariant.  In order to eliminate this (minor) difficulty, we will follow \cite{tz} and work instead in the cyclic group $X := \Z/N\Z$ using the obvious embedding $\iota: [N] \to X$.  We give this space the uniform Haar probability measure, thus
$$ \int_X f := \frac{1}{N} \sum_{x \in X} f(x)$$
for any $f: X \to \R$.  We also define shift maps $T^h f: X \to \R$ for any $h \in \Z$ by
$$ T^h f(x) := f(x+h);$$
clearly we have the identities
\begin{eqnarray*}
T^h T^k f &= T^{h+k} f \\
T^h(fg) &= (T^h f) (T^h g)\\
\int_X T^h f &= \int_X f
\end{eqnarray*}
for any $h,k \in \Z$ and $f,g: X \to \R$.  We will use these identities frequently without further comment in the sequel.

We will need a fixed quantity $\epsilon_0 > 0$ (depending only on $k,P_1,\dots,P_k$) to be chosen later.
We define the function $f: X \to \R$ by the formula
\begin{equation}\label{fdef}
 f(\iota(n)) = \frac{\epsilon_0}{10} \frac{\phi(W) \log N}{W} 1_{[N-\sqrt{N}] \backslash [\sqrt{N}]}(n) 1_A(nW+b)
\end{equation}
for $n \in [N]$, where $1_A$ denotes the indicator function of $A$; the reason for the normalizing factor $\frac{\epsilon_0}{10}$ is so that $f$ can be dominated by an enveloping sieve $\nu$, to be defined later.  We then have
\begin{equation}\label{fx}
 \int_X f \gg 1,
\end{equation}
where we allow the implied constants here to depend on the fixed quantities $\delta$ and $\epsilon_0$.

Now let $L > 0$ be a sufficiently large fixed quantity (depending on $k,P_1,\dots,P_k,\delta$) to be chosen later.  We will need the ``coarse scale''
\begin{equation}\label{m-def}
M := \log^L N
\end{equation}
which will basically be the domain of interest for the polynomials $P_1,\dots,P_k$, and the ``fine scale''
$$ H := \log^{\sqrt{L}} N,$$
which will basically be the scale used for various applications of the van der Corput inequality.  Note in particular that any quantity of size $O( H^{O(1)} \sqrt{M} )$ will also be $o(M)$ if the implied constants do not depend on $L$, and $L$ is large enough.

By hypothesis, if $N$ (or $N'$) is large enough, then the set $A$ contains no polynomial progressions of the form $a+P_1(Wm), \dots, a+P_k(Wm)$ with $m \in [M]$.  In particular, we see that
\begin{equation}\label{lff}
 \Lambda(f,\ldots,f) = 0
\end{equation}
where $\Lambda$ is the $k$-linear form
\begin{equation}\label{lambda-def}
\Lambda(f_1,\dots,f_k) := \E_{m \in [M]} \int_X T^{P_1(Wm)/W} f_1 \dots T^{P_k(Wm)/W} f_k
\end{equation}
for $k$ functions $f_1,\dots,f_k: X \to \R$, and we use the averaging notation
$$ E_{a \in A} f(a) := \frac{1}{|A|} \sum_{a \in A} f(a).$$
Note that there are no ``wraparound'' issues caused by the embedding into $\Z/N\Z$, due to our removal in \eqref{fdef} of the elements $n \in [N]$ that are less than $\sqrt{N}$ or larger than $N-\sqrt{N}$.  

On the other hand, using the (multi-dimensional) polynomial Szemer\'edi theorem of Bergelson and Leibman \cite{bl}, we have the following quantitative version of Therem \ref{pzt}:

\begin{theorem}[Quantitative Polynomial Szemer\'edi theorem]\label{qpst}  Let $\delta > 0$ be fixed, and let $g: X \to \R$ obey the pointwise bounds
$$ 0 \leq g \leq 1$$
(i.e., $0 \leq g(x) \leq 1$ for all $x \in X$), as well as the integral lower bound $\int_X g \geq \delta - o(1)$.  Then we have
$$ \Lambda( g, \dots, g ) \geq c(\delta) - o(1)$$
where $c(\delta)>0$ depends only on $\delta$, $k$, and $P_1,\dots,P_k$.
\end{theorem}

\begin{proof} See \cite[Theorem 3.2]{tz}. In that theorem, the common value $P_1(0)=\dots=P_k(0)$ of the $P_i$ was assumed to be zero, but the general case follows from this case by a simple change of variables.
\end{proof}

This theorem cannot be directly applied to control $\Lambda(f,\dots,f)$, because $f$ is not uniformly bounded.  However, Theorem \ref{szpp-short} will now be a consequence of the following claim.

\begin{theorem}[Approximation by bounded function]\label{abf}  Suppose that $\epsilon_0>0$ is sufficiently small (depending on $k,P_1,\dots,P_k$).  Let $\eps>0$ be fixed.  Suppose that $L$ is a fixed quantity which is sufficiently large depending on $k,P_1,\dots,P_k,\eps,\epsilon_0$. Let $f$ be as in \eqref{fdef}.  Then there exists $g: X \to \R$ with the pointwise bounds
$$ 0 \leq g \leq 1,$$
such that
\begin{equation}\label{fg}
 \left|\int_X f - \int_X g\right| \ll \eps + o(1)
\end{equation}
and
\begin{equation}\label{gg}
 |\Lambda(f,\dots,f) - \Lambda(g,\dots,g)| \ll \eps + o(1),
\end{equation}
where the implied constants in the $\ll$ notation do not depend on $\eps$ or $L$.
\end{theorem}

Let us assume Theorem \ref{abf} for now, and see how it implies Theorem \ref{szpp-short}.  Let $\epsilon_0>0$ be small enough for Theorem \ref{abf} to apply, $\eps > 0$ be a sufficiently small fixed quantity (depending on $\delta, \epsilon_0$) to be chosen later, and let $L$ be large enough depending on $k,P_1,\dots,P_k,\eps,\epsilon_0$ (in particular, $L$ will depend on $\delta$).  Let $g$ be as in Theorem \ref{abf}, let $M$ be defined by \eqref{m-def}, and $\Lambda$ defined by \eqref{lambda-def} (in particular, $\Lambda$ depends on $L$).  From \eqref{fx}, \eqref{fg}, and the triangle inequality we have
$$ \int_X g \gg 1$$
if $\eps$ is small enough depending on $\epsilon_0,\delta$.  Applying Theorem \ref{qpst}, we have
$$ \Lambda(g,\dots,g) \gg c_0 - o(1)$$
for some $c_0>0$ depending on $\epsilon_0,\delta$ but not on $\eps,L$.  By \eqref{gg} and the triangle inequality, we thus have
$$\Lambda(f,\dots,f) > 0$$
if $\eps$ is small enough (depending on $\epsilon_0,\delta$) and $N$ is large enough, contradicting \eqref{lff}; and Theorem \ref{szpp-short} follows.

It remains to establish Theorem \ref{abf}.  This is the focus of the remaining sections of the paper.

\begin{remark} The above arguments show that if one could remove the dependence of $L$ on $\eps$ in Theorem \ref{abf}, then one could also remove the dependence of $L$ on $\delta$ in Theorem \ref{szpp-short}.
\end{remark}

\section{The enveloping sieve}

As in \cite{gt-primes}, \cite{tz}, we view $f$ as a ``positive density fraction'' of a well-controlled \emph{enveloping sieve} $\nu$, which is defined by the explicit formula
$$ \nu(\iota(n)) := \frac{\phi(W) \log R}{W} \left( \sum_{d|Wn+b} \mu(d) \chi\left( \frac{\log d}{\log R} \right) \right)^2$$
for $n\in [N]$, where 
\begin{equation}\label{R-def}
R := N^{\epsilon_0},
\end{equation}
$\mu$ is the M\"obius function, and $\chi: \R \to \R$ is a fixed smooth even function supported on $[-1,1]$ with the normalization
$$ \int_0^1 |\chi'(t)|^2\ dt = 1$$
(where $\chi'$ is the derivative of $\chi$) and such that $\chi(0) \geq 1/2$ (say).  By construction and \eqref{fdef}, we have the pointwise bound
\begin{equation}\label{fnu}
 0 \leq f \leq \nu,
\end{equation}
that is to say that $0 \leq f(n) \leq \nu(n)$ for all $n \in X$.

From \cite[Corollary 10.5]{tz} we have the crude bound
\begin{equation}\label{joe}
\int_X \prod_{j = 1}^J T^{h_j} \nu = 1 + o(1) + O\left( \Exp\left( O\left( \sum_{1 \leq j < j' \leq J} \sum_{w \leq p \leq R^{\log R}: p | h_j - h_{j'}} \frac{1}{p} \right) \right) \right)
\end{equation}
for any fixed $J$ and any integers $h_1,\dots,h_J = O(\sqrt{N})$ (not necessarily distinct), assuming that $\epsilon_0$ is sufficiently small depending on $J$, and where $\Exp(x) := e^x - 1$, where the implied constant in the $O()$ exponent only depends on $J$.  Here, of course, we use $p|n$ to denote the assertion that $p$ divides $n$.  In particular, we have the mean bound
\begin{equation}\label{nu-stomp}
\int_X \nu = 1 + o(1)
\end{equation}
and the the crude bound
\begin{equation}\label{h1d}
\int_X T^{h_1} \nu \dots T^{h_J} \nu \ll \log^{O(1)} N
\end{equation}
for any fixed $J$ and any integers $h_1,\dots,h_J = O( \sqrt{N} )$ (not necessarily distinct), assuming $\epsilon_0$ is sufficiently small depending on $J$, and where the implied constant in the $O(1)$ exponent depends only on $J$.

One can use \eqref{joe} to establish the following fundamental pseudorandomness property:

\begin{proposition}[Polynomial forms condition]\label{pfc}  Let $J, d, D \geq 1$ be fixed natural numbers.  Suppose that $\epsilon_0$ is sufficiently small depending on $J,d,D$, and that $L$ is sufficiently large depending on $J,d,D$.  Then for any polynomials $Q_1,\dots,Q_J \in \Z[\mathbf{m}_1,\dots,\mathbf{m}_d]$ of degree at most $D$, with coefficients of size $O(W^{O(1)})$, and with $Q_j-Q_{j'}$ non-constant for every $1 \leq j < j' \leq d$, and any convex body $\Omega \in [-M^2,M^2]^d$ of inradius\footnote{The \emph{inradius} of a convex body is the radius of the largest open ball one can inscribe inside the body.} at least $H^{1/2}$, we have the asymptotic
\begin{equation}\label{qjf}
 \E_{\vec h \in \Omega \cap \Z^d} \int_X \prod_{j = 1}^J T^{Q_j(\vec h)} \nu = 1 + o(1).
\end{equation}
\end{proposition}

This proposition was established in \cite[Theorem 3.18]{tz} in the case when $M$ is a small power of $N$; the point is that $M$ can be lowered to essentially $\log^L N$ for some large $L$.  However, note that the number $J$ of polynomials involved cannot be arbitrarily large depending on $\epsilon_0$.  The main obstruction to reducing the size of the coarse scale $M$ is that we need the ``diagonal'' contributions to \eqref{qjf} (such as those coming from the terms where one of the $Q_j(\vec h)$ vanish) to be negligible when compared to the remaining terms.  The sieve $\nu$ (or powers thereof, such as $\nu^2$) tends to have size $\log^{O(1)} N$ on the average, and using this one expects to control diagonal contributions to \eqref{qjf} by something like $\log^{O(1)} N / M$, which will be negligible when $M$ is a sufficiently large power of $\log N$.

\begin{proof}  We repeat the arguments from \cite[\S 11]{tz}.  For each $\vec h$, we see from \eqref{joe} that
$$
\int_X \prod_{j = 1}^J T^{Q_j(\vec h)} \nu = 1 + o(1) + O\left( \Exp\left( O\left( \sum_{1 \leq j < j' \leq J} \sum_{w \leq p \leq R^{\log R}: p | Q_j(\vec h)-Q_{j'}(\vec h)} \frac{1}{p} \right) \right) \right).$$
Thus it suffices to show that
$$ \E_{\vec h \in \Omega \cap \Z^d} 
\Exp\left( O\left( \sum_{1 \leq j < j' \leq J} \sum_{w \leq p \leq R^{\log R}: p | Q_j(\vec h)-Q_{j'}(\vec h)} \frac{1}{p} \right) \right) = o(1).$$
Using the elementary bound $\Exp(a+b) \ll \Exp(2a) + \Exp(2b)$ repeatedly, it suffices to show that
$$ \E_{\vec h \in \Omega \cap \Z^d} 
\Exp\left( O\left( \sum_{w \leq p \leq R^{\log R}: p | Q_j(\vec h)-Q_{j'}(\vec h)} \frac{1}{p} \right) \right) = o(1).$$
 for each $1 \leq j < j \leq d'$.  

We first dispose of the ``globally bad'' primes, in which $p$ divides the entire polynomial $Q_j-Q_{j'}$.  As $Q_j-Q_{j'}$ is non-constant and has coefficients $O(W^{O(1)})$, we see that the product of all such primes is $O(W^{O(1)})$.  In \cite[Lemma E.3]{tz}, it is shown that
$$ \sum_{p \geq w: p|n} \frac{1}{p} = o(1)$$
for any $n = O(W^{O(1)})$.  Thus the contribution of such primes in the above sum is negligible.

In \cite[Lemma E.1]{tz}, it is shown that
$$ \Exp\left( O\left( \sum_{p \in A} \frac{1}{p} \right) \right) \ll \sum_{p \in A} \frac{\log^{O(1)} p}{p}$$
for any set $A$ of primes.  Thus it suffices to show that
$$
\sum_{w \leq p \leq R^{\log R}} \frac{\log^{O(1)} p}{p} \E_{\vec h \in \Omega \cap \Z^d} 1_{p|Q_j(\vec h)-Q_{j'}(\vec h); p \not | Q_j - Q_{j'}} = o(1).
$$
From \cite[Lemma D.3]{tz}, we conclude that if $p$ does not divide $Q_j - Q_{j'}$, then the average $1_{p|Q_j(\vec h)-Q_{j'}(\vec h)}$ on any cube in $\Z^d$ of sidelength $1 \leq K \leq p$ is $O( \frac{1}{K} )$.   By \cite[Corollary C.2]{tz}, \cite[Lemma C.4]{tz} and the inradius hypothesis we then have
$$
\E_{\vec h \in \Omega \cap \Z^d} 1_{p|h_j-h_{j'}} \ll \frac{1}{p} + H^{-1/2}. 
$$ 
Thus we reduce to showing that
$$
\sum_{w \leq p \leq R^{\log R}} \frac{\log^{O(1)} p}{p^2} + H^{-1/2} \frac{\log^{O(1)} p}{p} = o(1),
$$
but this follows easily from Mertens' theorem (or the prime number theorem) and the definition of $M$ and $w$, if $L$ is large enough.
\end{proof}

\section{Averaged local Gowers norms}

As in \cite{tz}, we will control the left-hand side of \eqref{gg} using some local Gowers norms, which we now define.  Given any $d \geq 2$ and integers $a_1,\dots,a_d$ and any scale $S \geq 1$, we define the \emph{local Gowers uniformity norms} $U^{a_1,\dots,a_d}_S$ by the formula
\begin{equation}\label{usdef}
\begin{split}
\| f \|_{U^{a_1,\dots,a_d}_S}^{2^d} &:= \E_{m_1^{(0)},\dots,m_d^{(0)}, m_1^{(1)},\dots,m_d^{(1)} \in [S]} \int_X \\
&\quad \prod_{(\omega_1,\dots,\omega_d) \in \{0,1\}^d} T^{m_1^{(\omega_1)} a_1 + \dots + m_d^{(\omega_d)} a_d} f
\end{split}
\end{equation}
for $f: X \to \R$.  Next, for any $t \geq 2$ and any $d$-tuple $\vec Q = (Q_1,\dots,Q_d)$ of polynomials $Q_i \in \Z[ \mathbf{h}_1, \dots, \mathbf{h}_t, \mathbf{W}]$ in $t+1$ variables, we define the \emph{averaged local Gowers uniformity norms} $U^{\vec Q([H]^t,W)}_{S}$ on functions $f: X \to \R$ by the formula
\begin{equation}\label{uqs-def}
\| f \|_{U^{\vec Q([H]^t,W)}_S}^{2^d} := \E_{\vec h \in [H]^t} \|f\|_{U^{Q_1(\vec h,W),\dots,Q_d(\vec h,W)}_S}^{2^d}.
\end{equation}
These are indeed norms; see \cite[Appendix A]{tz}.  One can extend these norms to complex-valued functions by inserting an alternating sequence of conjugation symbols in the product, but we will not need to use such an extension here.   We remark that these expressions may also be defined for $d=1$, but are merely seminorms instead of norms in that case.

From the Gowers-Cauchy-Schwarz inequality (see e.g. \cite[Appendix B]{linprimes}) and H\"older's inequality, we record the useful inequality
\begin{equation}\label{gcs}
\begin{split}
&\left|\E_{m_1^{(0)},\dots,m_d^{(0)}, m_1^{(1)},\dots,m_d^{(1)} \in [S]} \int_X \prod_{\omega \in \{0,1\}^d} T^{m_1^{(\omega_1)} a_1 + \dots + m_d^{(\omega_d)} a_d} f_\omega\right|\\
&\quad \leq \prod_{\omega \in \{0,1\}^d} \|f_\omega\|_{U^{a_1,\dots,a_d}_S}^{2^d}
\end{split}
\end{equation}
for any functions $f_\omega: X \to \R$ for $\omega\in \{0,1\}^d$ where we write $\omega := (\omega_1,\dots,\omega_d)$.

In a similar spirit, we have the following basic inequality:

\begin{theorem}[Polynomial generalized von Neumann theorem]\label{pgvn}  Suppose that $\epsilon_0>0$ is a fixed quantity which is sufficiently small depending on $k,P_1,\dots,P_k$, and that $L$ is a fixed quantity which is sufficiently large depending on $k,P_1,\dots,P_k$. 
Then there exists fixed $t \geq 0, d \geq 2$ and a fixed $d$-tuple $\vec Q = (Q_1,\dots,Q_d)$ of polynomials $Q_i \in \Z[ \mathbf{h}_1, \dots, \mathbf{h}_t, \mathbf{W}]$, none of which are identically zero, and which are independent of $\epsilon_0, L$, such that one has the inequality
$$
\left| \Lambda(g_1,\dots,g_k) \right|  \ll \min_{1 \leq i \leq k} \| g_i \|_{U^{\vec Q([H]^t,W)}_{\sqrt{M}}}^c + o(1)
$$
for some fixed $c>0$ and all $g_1,\dots,g_k: X \to \R$ obeying the pointwise bound $|g_i| \leq \nu+1$ for all $i=1,\dots,d$.
\end{theorem}

\begin{proof}  This is essentially \cite[Theorem 4.5]{tz} (which was proven by a combination of PET induction, the Cauchy-Schwarz inequality, and the polynomial forms condition), with the only difference being that $H$ and $M$ are now polylogarithmic in $N$, rather than polynomial in $N$.  However, an inspection of the proof of \cite[Theorem 4.5]{tz} shows that this does not affect the arguments (after replacing the polynomial forms condition used there with Proposition \ref{pfc}, of course); the key relationships between $H,M,N$ that are used in the proof are that $(HW)^{O(1)} \sqrt{M} = o(M)$ and that $(HWM)^{O(1)} = o(\sqrt{N})$, where the implied constants depend only on $k,P_1,\dots,P_k$ (and in particular are independent of $\epsilon_0, L$), and these properties are certainly obeyed for the choice of $H$ and $M$ used here.  (The bound \eqref{h1d} is sufficient to deal with all the error terms arising from use of the van der Corput inequality in this regime.)  
\end{proof}

Setting $g_2=\dots=g_k=1$, we obtain in particular that
$$ \left|\int_X g_1\right| \ll \|g_1\|_{U^{\vec Q([H]^t,W)}_{\sqrt{M}}}^c + o(1).$$
(In fact, one can take $c=1$ in this inequality by the standard monotonicity properties of the Gowers norms, see \cite[Lemma A.3]{tz}, but we will not need this improvement here.)

In view of this theorem, Theorem \ref{abf} is now a consequence of the following claim (after replacing $\eps$ with $\eps^c$):

\begin{theorem}[Approximation by bounded function, again]\label{abfa}  Let $d \geq 2$ and $t \geq 0$ be fixed, and let $\vec Q = (Q_1,\dots,Q_d)$ be a fixed $d$-tuple of polynomials $Q_i \in \Z[ \mathbf{h}_1, \dots, \mathbf{h}_t, \mathbf{W}]$, not identically zero.  Let $\epsilon_0>0$ be a fixed quantity that is sufficiently small depending on $d,t,\vec Q$.  Let $\eps > 0$ be fixed, and let $L$ be a fixed quantity that is sufficiently large depending on $d,t,\vec Q,\epsilon_0,\eps$.  Let $\nu$ be as above, and let $f: X \to \R$ obey the pointwise bound
$$ 0 \leq f \leq \nu.$$
Then there exists $g: X \to \R$ with the pointwise bound
$$ 0 \leq g \leq 1,$$
such that
\begin{equation}\label{fgu}
 \|f-g\|_{U^{\vec Q([H]^t,W)}_{\sqrt{M}}} \ll \eps + o(1).
\end{equation}
\end{theorem}

It remains to establish Theorem \ref{abfa}.  This is the objective of the remaining sections of the paper.

\begin{remark} As before, if one could remove the dependence of $L$ on $\eps$ in Theorem \ref{abfa}, then one could remove the dependence of $L$ on $\delta$ in Theorem \ref{szpp-short}.  Also, from this point on the number $k$ of polynomials $P_1,\dots,P_k$ in Theorem \ref{szpp-short} no longer plays a role, and we will use the symbol $k$ to denote other (unrelated) natural numbers.
\end{remark}

\section{The dense model theorem}

Let $d,t,\vec Q$ be as in Theorem \ref{abfa}.  The averaged local Gowers norm $\| f \|_{U^{\vec Q([H]^t,W)}_{\sqrt{M}}}$ of a function $f: X \to \R$ can then be expressed by the identity
\begin{equation}\label{udual}
\| f \|_{U^{\vec Q([H]^t,W)}_{\sqrt{M}}}^{2^d} = \int f {\mathcal D} f
\end{equation}
where the \emph{dual function} ${\mathcal D} f = {\mathcal D}^{\vec Q([H]^t,W)}_{\sqrt{M}} f: X \to \R$ is defined by the formula
\begin{equation}\label{ddef}
\begin{split}
{\mathcal D} f &:= \E_{\vec h \in [H]^t} \E_{m_1^{(0)},\dots,m_d^{(0)}, m_1^{(1)},\dots,m_d^{(1)} \in [\sqrt{M}]} \\
&\quad \prod_{(\omega_1,\dots,\omega_d) \in \{0,1\}^d \backslash \{0\}^d} T^{\sum_{i=1}^d (m_i^{(\omega_i)} - m_i^{(0)}) Q_i(\vec h)} f
\end{split}
\end{equation}
More generally, we define
\begin{equation}\label{ddef-gen}
\begin{split}
{\mathcal D} (f_\omega)_{\omega \in \{0,1\}^d \backslash \{0\}^d} &:= \E_{\vec h \in [H]^t} \E_{m_1^{(0)},\dots,m_d^{(0)}, m_1^{(1)},\dots,m_d^{(1)} \in [\sqrt{M}]} \\
&\quad \prod_{(\omega_1,\dots,\omega_d) \in \{0,1\}^d \backslash \{0\}^d} T^{\sum_{i=1}^d (m_i^{(\omega_i)} - m_i^{(0)}) Q_i(\vec h)} f_\omega
\end{split}
\end{equation}
for any tuple of functions $f_\omega: X \to \R$ for $\omega \in \{0,1\}^d \backslash \{0\}^d$.

Theorem \ref{abfa} is then an immediate consequence of combining the following two theorems (with the function $f$ appearing in Theorem \ref{dens} being replaced by the function $f-g$ appearing in Theorem \ref{abfa-weak}).

\begin{theorem}[Weak approximation by bounded function]\label{abfa-weak}  Let $d \geq 2$ and $t \geq 0$ be fixed, and let $\vec Q = (Q_1,\dots,Q_d)$ be a fixed $d$-tuple of polynomials $Q_i \in \Z[ \mathbf{h}_1, \dots, \mathbf{h}_t, \mathbf{W}]$, not identically zero.  Let $\epsilon_0>0$ be a fixed quantity that is sufficiently small depending on $d,t,\vec Q$.  Let $\eps > 0$ be fixed, and let $L$ be a fixed quantity that is sufficiently large depending on $d,t,\vec Q,\epsilon_0,\eps$.  Let $\nu$ be as above, and let $f: X \to \R$ obey the pointwise bound
$$ 0 \leq f \leq \nu.$$
Then there exists $g: X \to \R$ with the pointwise bound
$$ 0 \leq g \leq 1,$$
such that
\begin{equation}\label{fgu-weak}
 \left|\int_X (f-g) {\mathcal D} (F_\omega)_{\omega \in \{0,1\}^d \backslash \{0\}^d}\right| \ll \eps + o(1)
\end{equation}
for all functions $F_\omega: X \to \R$ for $\omega \in \{0,1\}^d \backslash \{0\}^d$ with the pointwise bounds $-1 \leq F_\omega \leq 1$.
\end{theorem}

\begin{theorem}[Densification]\label{dens} Let $d \geq 2$ and $t \geq 0$ be fixed, and let $\vec Q = (Q_1,\dots,Q_d)$ be a fixed $d$-tuple of polynomials $Q_i \in \Z[ \mathbf{h}_1, \dots, \mathbf{h}_t, \mathbf{W}]$, not identically zero.  Let $\epsilon_0>0$ be a fixed quantity that is sufficiently small depending on $d,t,\vec Q$.  Let $\eps > 0$ be fixed, and let $L$ be a fixed quantity that is sufficiently large depending on $d,t,\vec Q,\epsilon_0$.  Let $\nu$ be as above, and let $f: X \to \R$ obey the pointwise bound
$$ |f| \leq \nu+1,$$
and suppose that
\begin{equation}\label{dons}
 \left|\int_X f {\mathcal D} (F_\omega)_{\omega \in \{0,1\}^d \backslash \{0\}^d}\right| \ll \eps + o(1) 
\end{equation}
for all functions $F_\omega: X \to \R$ with the pointwise bounds $-1 \leq F_\omega \leq 1$.  Then we have
$$ \| f \|_{U^{\vec Q([H]^t,W)}_{\sqrt{M}}} \ll \eps^c + o(1)$$
for some fixed $c>0$ (independent of $\eps$).
\end{theorem}

We prove Theorem \ref{abfa-weak} in this section, and Theorem \ref{dens} in the next section.

To prove Theorem \ref{abfa-weak}, we invoke the dense model theorem, first established implicitly in \cite{gt-primes} and then made more explicit in \cite{tz}, \cite{gowers}, \cite{reingold}.  We use the formulation from \cite[Theorem 1.1]{reingold}:

\begin{theorem}[Dense model theorem]  For every $\eps > 0$, there is $K = (1/\eps)^{O(1)}$ and $\eps' > 0$ such that, whenever ${\mathcal F}$ is a set of bounded functions from $X$ to $[-1,1]$, and $\nu: X \to \R^+$ obeys the bound
$$ |\int_X (\nu-1) F_1 \dots F_{K'}| \leq \eps'$$
for all $0 \leq K' \leq K$ and $F_1,\dots,F_{K'} \in {\mathcal F}$, and every function $f: X \to \R$ with $0 \leq f \leq \nu$, one has a function $g: X \to [0,1]$ such that 
$$ |\int_X (f-g) F| \leq \eps$$
for all $F \in {\mathcal F}$.
\end{theorem}

This reduces Theorem \ref{abfa-weak} to the following calculation:

\begin{theorem}[Orthogonality to dual functions]\label{orthodual}  Let $d \geq 2$ and $t \geq 0$ be fixed, and let $\vec Q = (Q_1,\dots,Q_d)$ be a fixed $d$-tuple of polynomials $Q_i \in \Z[ \mathbf{h}_1, \dots, \mathbf{h}_t, \mathbf{W}]$, not identically zero.  Let $\epsilon_0>0$ be a fixed quantity that is sufficiently small depending on $d,t,\vec Q$.  Let $K > 0$ be a fixed integer, and let $L$ be a fixed quantity that is sufficiently large depending on $d,t,\vec Q,\epsilon_0,K$.  Let $\nu$ be as above.  Then one has
\begin{equation}\label{dual-ortho}
\int_X (\nu-1) \prod_{k=1}^K {\mathcal D} (F_{k,\omega})_{\omega \in \{0,1\}^d \backslash \{0\}^d} = o(1)
\end{equation}
for all functions $F_{k,\omega}: X \to \R$ for $k=1,\dots,K$, $\omega \in \{0,1\}^d \backslash \{0\}^d$ with the pointwise bounds $-1 \leq F_{k,\omega} \leq 1$.
\end{theorem}

\begin{remark}\label{delta} The fact that $L$ depends on $K$ here is the sole reason why $L$ depends on $\delta$ in Theorem \ref{szpp-short} (note that no parameter related to $\delta$ or $K$ appears in Theorem \ref{dens}).
\end{remark}

We now prove Theorem \ref{orthodual}.  Let $d,t,\vec Q, \epsilon_0, K, L, \nu,F_{k,\omega}$ be as in that theorem.  We expand out the left-hand side of \eqref{dual-ortho} as the average of
\begin{equation}\label{num}
\begin{split}
&\int_X (\nu-1) \E_{m_{i,k}^{(\omega)} \in [\sqrt{M}] \forall i=1,\dots,d; k=1,\dots,K; \omega=0,1}  \\
&\quad \prod_{(\omega_1,\dots,\omega_d) \in \{0,1\}^d \backslash \{0\}^d} \prod_{k=1}^K T^{\sum_{i=1}^d (m_{i,k}^{(\omega_i)} - m_i^{(0)}) Q_i(\vec h_k)} F_{k,\omega}
\end{split}
\end{equation}
as $\vec h_1,\dots,\vec h_K$ ranges over $[H]^t$.

We first deal with the degenerate cases in which $Q_i(\vec h_k)=0$ for some $i,k$.  By the Schwarz-Zippel lemma (see e.g. \cite[Lemma D.3]{tz}), the number of tuples $(\vec h_1,\dots,\vec h_K)$ with this degeneracy is $O( H^{Kt-1} )$.  Meanwhile, from \eqref{nu-stomp} and the boundedness of the $F_{k,\omega}$, each expression \eqref{num} is $O(1)$.  Thus the total contribution of this case is $O(H^{-1})$, which is acceptable.

It thus suffices to show that the expression \eqref{num} is $o(1)$ uniformly for all $\vec h_1,\dots, \vec h_K$ with none of the $Q_i(\vec h_k)$ vanishing.

The next step is to ``clear denominators'' (as in \cite{tz}).  Fix $\vec h_1,\dots,\vec h_K$, and write $D_i := \prod_{k=1}^K |Q_i(\vec h_k)|$ for $i=1,\dots,d$.  Then we have $1 \leq D_i \ll O( HW )^{O(K)}$, and we can write
$$ D_i = Q_i(\vec h_k) r_{i,k}$$
for each $i=1,\dots,d$, $k=1,\dots,K$, and some $r_{i,k} = O( HW )^{O(K)}$.  

Let $n_1^{(0)},\dots,n_d^{(0)},n_1^{(1)},\dots,n_d^{(1)}$ be elements of $[M^{1/4}]$.  Then if we shift each variable $m_{i,k}^{(\omega)}$ by $r_{i,k} n_i^{(\omega)}$, we can rewrite \eqref{num} as
\begin{equation}\label{num-2}
\begin{split}
&\int_X (\nu-1) \E_{m_{i,k}^{(\omega)} \in [\sqrt{M}] - r_{i,k} n_i^{(\omega)} \forall i=1,\dots,d; k=1,\dots,K; \omega=0,1}  \\
&\quad \prod_{(\omega_1,\dots,\omega_d) \in \{0,1\}^d \backslash \{0\}^d} \prod_{k=1}^K T^{\sum_{i=1}^d (m_{i,k}^{(\omega_i)} - m_i^{(0)}) Q_i(\vec h_k) + (n_i^{(\omega_i)}-n_i^{(0)}) D_i} F_{k,\omega}.
\end{split}
\end{equation}
The shifted interval $[\sqrt{M}] - r_{i,k} n_i^{(\omega)}$ differs from $[\sqrt{M}]$ by shifts by a set of cardinality $O( M^{1/4} (HW)^{O(K)} )$, and so by \eqref{h1d} one can replace the former by the latter after accepting an additive error of $O( (\log^{O(1)} N) M^{1/4} (HW)^{O(1)} / \sqrt{M} )$, where the implied constants in the exponents depend on $K$.  It is at this point that we crucially use the hypothesis that $L$ be large compared with $K$, to ensure that this error is still $o(1)$.  Thus \eqref{num-2} can be written as
\begin{eqnarray*}
&\int_X (\nu-1) \E_{m_{i,k}^{(\omega)} \in [\sqrt{M}] \forall i=1,\dots,d; k=1,\dots,K; \omega=0,1}   \\
&\quad \prod_{(\omega_1,\dots,\omega_d) \in \{0,1\}^d \backslash \{0\}^d} \prod_{k=1}^K T^{\sum_{i=1}^d (m_{i,k}^{(\omega_i)} - m_i^{(0)}) Q_i(\vec h_k)  + (n_i^{(\omega_i)}-n_i^{(0)}) D_i} F_{k,\omega}\\
&\quad + o(1).
\end{eqnarray*}
Averaging over all such $n_i^{(\omega)}$, we obtain
\begin{eqnarray*}
&\E_{m_{i,k}^{(\omega)} \in [\sqrt{M}] \forall i=1,\dots,d; k=1,\dots,K; \omega=0,1} \int_X (\nu-1) \E_{n_1^{(0)},\dots,n_d^{(0)},n_1^{(1)},\dots,n_d^{(1)} \in [M^{1/4}]} \\
&\quad \prod_{(\omega_1,\dots,\omega_d) \in \{0,1\}^d \backslash \{0\}^d} \prod_{k=1}^K T^{\sum_{i=1}^d (m_{i,k}^{(\omega_i)} - m_i^{(0)}) Q_i(\vec h_k)  + (n_i^{(\omega_i)}-n_i^{(0)}) D_i} F_{k,\omega}\\
&\quad + o(1).
\end{eqnarray*}
Shifting the integral $\int_X$ by $\sum_{i=1}^d (n_i^{(0)}) D_i$ and then using the Gowers-Cauchy-Schwarz inequality \eqref{gcs} (and the boundedness of the functions $\prod_{k=1}^K T^{\sum_{i=1}^d (m_{i,k}^{(\omega_i)} - m_i^{(0)}) Q_i(\vec h_k)} F_{k,\omega}$), we may bound this by
$$ \| \nu-1 \|_{U^{D_1,\dots,D_d}_{M^{1/4}}} + o(1).$$
But from expanding out the Gowers norm \eqref{usdef} and using Proposition \ref{pfc} to estimate the resulting $2^{2^d}$ terms (cf. \cite[Lemma 5.2]{gt-primes}), we see that 
\begin{equation}\label{nud}
 \| \nu-1 \|_{U^{D_1,\dots,D_d}_{M^{1/4}}}^{2^d} = o(1)
\end{equation}
and Theorem \ref{orthodual} follows.

\section{Densification}

Now we prove Theorem \ref{dens}.  It will suffice to establish the following claim.

\begin{proposition}  Let the notation and hypotheses be as in Theorem \ref{dens}.  Then one has
\begin{equation}\label{bound}
\begin{split}
& \left|\E_{\vec h \in [H]^t} \E_{m_1^{(0)},\dots,m_d^{(0)}, m_1^{(1)},\dots,m_d^{(1)} \in [\sqrt{M}]} \int_X 
\prod_{\omega \in \{0,1\}^d} T^{m_1^{(\omega_1)} Q_1(\vec h) + \dots + m_d^{(\omega_d)} Q_d(\vec h)} f_\omega\right| \\
&\quad \ll \eps^c + o(1)
\end{split}
\end{equation}
for some fixed $c>0$ (independent of $\eps$), whenever $(f_\omega)_{\omega \in \{0,1\}^d}$ is a tuple of functions $f_\omega: X \to \R$ (with $\omega := (\omega_1,\dots,\omega_d)$), such that one of the $f_\omega$ is equal to $f$, and each of the remaining functions $f_\omega$ in the tuple either obey the pointwise bound $|f_\omega| \leq 1$ or $|f_\omega| \leq \nu$.
\end{proposition}

Indeed, given the above proposition, then by triangle inequality and decomposition we may replace the bounds $|f_\omega| \leq 1$ or $|f_\omega| \leq \nu$ with $|f_\omega| \leq \nu+1$, and then by setting $f_\omega=f$ for every $\omega$ and using \eqref{udual}, we obtain the claim.

It remains to prove the proposition.  We induct on the number of factors $f_\omega$ for which one has the bound $|f_\omega| \leq \nu$ instead of $|f_\omega| \leq 1$.  First suppose that there are no such factors, thus $|f_\omega| \leq 1$ for all $\omega$ except for one $\omega$, for which $f_\omega=f$.  By permuting the cube $\{0,1\}^d$, we may assume that it is $f_{\{0\}^d}$ that is equal to $f$, with all other $f_\omega$ bounded in magnitude by $1$.  But then the expression in \eqref{bound} may be rewritten as
$$ \int_X f {\mathcal D} (f_\omega)_{\omega \in \{0,1\}^d \backslash \{0\}},$$
and the claim follows from the hypothesis \eqref{dons}.

Now suppose that at least one of the $f_\omega$ (other than the one equal to $f$) is bounded in magnitude by $\nu$ rather than $1$.  By permuting the cube we may assume that $|f_{\{0\}^d}| \leq \nu$.  We then write the left-hand side of \eqref{bound} as
$$ \int_X f_{\{0\}^d} {\mathcal D} \vec f$$
where $\vec f := (f_\omega)_{\omega \in \{0,1\}^d \backslash \{0\}}$.
By Cauchy-Schwarz, it thus suffices to show that
$$ \int_X \nu ({\mathcal D} \vec f)^2 \ll \eps^c + o(1)$$
for some fixed $c>0$.  We will split this into two estimates,
\begin{equation}\label{nuo}
 |\int_X (\nu-1) ({\mathcal D} \vec f)^2| = o(1)
\end{equation}
and
\begin{equation}\label{bland}
 \int_X ({\mathcal D} \vec f)^2 \ll \eps^c + o(1).
\end{equation}
We set aside \eqref{nuo} for now and work on \eqref{bland}.  Bounding all the components of $\vec f$ in magnitude by $\nu+1$ and using Proposition \ref{pfc}, we see that
$$
 \int_X ({\mathcal D} \vec f)^4 \ll 1,
$$
so by H\"older's inequality, it suffices to show that
$$
 \int_X |{\mathcal D} \vec f| \ll \eps^c + o(1)
$$
(for a possibly different fixed $c>0$).  It thus suffices to show that
$$
|\int_X g {\mathcal D} \vec f| \ll \eps^c + o(1)
$$
whenever $g: X \to \R$ is such that $|g| \leq 1$.  But this expression is of the form \eqref{bound} with $f_{\{0\}}$ replaced by $g$, and the claim then follows from the induction hypothesis.

It thus remains to show \eqref{nuo}.  We can rewrite
$$({\mathcal D} \vec f)^2 = {\mathcal D}^{\vec Q \oplus \vec Q([H]^t,W)}_{\sqrt{M}} \vec f_2$$
where $\vec Q \oplus \vec Q$ is the $2d$-tuple
$$\vec Q \oplus \vec Q = (Q_1,\dots,Q_d,Q_1,\dots,Q_d)$$
and $\vec f_2 = (f_{2,\omega})_{\omega \in \{0,1\}^{2d} \backslash \{0\}^d}$ is defined by setting
$$ f_{2,\omega \oplus \{0\}^d} := f_\omega$$
and
$$ f_{2,\{0\}^d \oplus \omega} := f_\omega$$
for $\omega \in \{0,1\}^d \backslash \{0\}^d$, and 
$$ f_{2,\omega \oplus \omega'} := 1$$
for $\omega,\omega' \in \{0,1\}^d \backslash \{0\}^d$.  Applying the Gowers-Cauchy-Schwarz inequality \eqref{gcs}, we may thus bound the left-hand side of \eqref{nuo} by
$$ \| \nu-1\|_{U^{\vec Q \oplus \vec Q([H]^t,W)}_{\sqrt{M}}} \prod_{\omega \in \{0,1\}^d \backslash \{0\}} \|f_\omega\|_{U^{\vec Q \oplus \vec Q([H]^t,W)}_{\sqrt{M}}}^2.$$
Bounding $f_\omega$ by $\nu$ or $1$ and using Proposition \ref{pfc}, we can bound
$$ \|f_\omega\|_{U^{\vec Q \oplus \vec Q([H]^t,W)}_{\sqrt{M}}} \ll 1$$
and further application of Proposition \ref{pfc} (cf. \eqref{nud}) gives
$$ \|\nu-1\|_{U^{\vec Q \oplus \vec Q([H]^t,W)}_{\sqrt{M}}} = o(1)$$
and the claim follows.

\section{The linear case}\label{linear-sec}

We now explain why in the linear case $P_i = (i-1){\mathbf m}$ of Theorem \ref{szpp-short}, one may take $L$ to be independent of $\delta$.  In the linear case, one can replace the averaged local Gowers norm $U^{\vec Q([H]^t,W)}_{\sqrt{M}}$ in Theorem \ref{pgvn} with the simpler norm $U^{1,\dots,1}_{\sqrt{M}}$, where $1$ appears $d=k-1$ times; this follows by repeating the proof of \cite[Proposition 5.3]{gt-primes}, after replacing some global averages with local ones.  (In fact one could replace $\sqrt{M}$ here by $M^{1-\sigma}$ for any fixed $\sigma>0$.) As such, we can ignore the $H$ parameter and the $\vec h$ averaging, and just prove Theorem \ref{orthodual} in the case when $Q_1=\dots=Q_d=1$.  Here, the ``clearing denominators'' step is unnecessary, and so $L$ does not need to be large depending on $K$, which by Remark \ref{delta} ensures that the final $L$ is independent of $\delta$.

\begin{remark} A more careful accounting of exponents (in particular, replacing \eqref{joe} with a more precise asymptotic involving a singular series similar to that in \eqref{grdef}) allows one to take $L$ as small as $C k 2^k$ for some absolute constant $C$; we omit the details.
\end{remark}

\begin{acknowledgement}
The first author is supported by a Simons Investigator grant, the James and Carol Collins Chair, the Mathematical Analysis \& Application Research Fund Endowment, and by NSF grant DMS-1266164. Part of this research was performed while the first author was visiting the Institute for Pure and Applied Mathematics (IPAM), which is supported by the National Science Foundation. The second author is supported by ISF grant 407/12. The second author was on sabbatical at Stanford while part of this work was carried out; she would like to thank the Stanford math department for its hospitality and support.  Finally, the authors thank the anonymous referee for a careful reading of the paper.
\end{acknowledgement}


\begin{thebibliography}{10}

\bibitem{benatar}
J. Benatar, \emph{The existence of small prime gaps in subsets of the integers}, preprint.

\bibitem{pet}
V. Bergelson, \emph{Weakly mixing PET}, Ergodic Theory Dynam. Systems \textbf{7} (1987), no. 3, 337--349. 

\bibitem{bl} V. Bergelson, A. Leibman, \emph{Polynomial extensions of van der Waerden's and Szemer\'edi's theorems. } J. Amer. Math. Soc. 9 (1996), no. \textbf{3}, 725--753.

\bibitem{bll}
V. Bergelson, A. Leibman, E. Lesigne, \emph{Intersective polynomials and the polynomial Szemer\'edi theorem}, Adv. Math. \textbf{219} (2008), no. 1, 369--388.



\bibitem{cfz}
D. Conlon, J. Fox, Y. Zhao, \emph{A relative Szemer\'edi theorem}, preprint.



\bibitem{cg}
D. Conlon, T. Gowers, \emph{Combinatorial theorems in sparse random sets}, preprint.



\bibitem{gallagher}
P. X. Gallagher, \emph{On the distribution of primes in short intervals}, Mathematika \textbf{23} (1976), 4--9.


\bibitem{gpy}
D. Goldston, J. Pintz, C. Yildirim, \emph{Primes in tuples IV: Density of small gaps between consecutive primes}, Acta Arith. \textbf{160} (2013), no. 1, 37--53.



\bibitem{gowers}
W. T. Gowers, \emph{Decompositions, approximate structure, transference, and the Hahn-Banach theorem}, Bull. Lond. Math. Soc. \textbf{42} (2010), no. 4, 573--606.


\bibitem{gt-primes}
B. Green, T. Tao, \emph{The primes contain arbitrarily long arithmetic progressions}, Ann. of Math. (2) \textbf{167} (2008), no. 2, 481--547.


\bibitem{linprimes}
B. Green, T. Tao, \emph{Linear equations in primes}, Annals of Math, \textbf{171} (2010), 1753--1850.



\bibitem{hr}
H. Halberstam, H.-E. Richert, Sieve methods, Academic Press, New York, 1974.


 \bibitem{hardy} G. H. Hardy, J. E. Littlewood, \emph{Some problems of
     ``Partitio Numerorum'', III: On the expression of a number as a
     sum of primes}, Acta Math. 44 (1923), 1--70.



\bibitem{le}
T. H. Le, \emph{Intersective polynomials and the primes}, J. Number Theory \textbf{130} (2010), no. 8, 1705--1717.

\bibitem{maynard}
J. Maynard, \emph{Small gaps between primes}, preprint. 


\bibitem{reingold}
O. Reingold, L. Trevisan, M. Tulsiani and S. Vadhan, \emph{Dense subsets of pseudorandom sets}, Electronic Colloquium on Computational Complexity, Proceedings of 49th IEEE FOCS, 2008.

\bibitem{schacht}
M. Schacht, \emph{Extremal results for random discrete structures}, preprint.

\bibitem{szemeredi}
E. Szemer\'edi, \emph{On sets of integers containing no k elements in arithmetic progression}, 
Acta Arith. \textbf{27} (1975), 299-345.



\bibitem{tz}
T. Tao, T. Ziegler, \emph{The primes contain arbitrarily long polynomial progressions}, Acta Math. \textbf{201} (2008), no. 2, 213--305


\bibitem{zhang}
Y. Zhang, \emph{Bounded gaps between primes}, to appear, Annals Math.

\bibitem{zhou}
B. Zhou, \emph{The Chen primes contain arbitrarily long arithmetic progressions}, Acta Arith. \textbf{138} (2009), no. 4, 301--315.

\end{thebibliography}
\end{document}